\documentclass[12pt]{amsart}

\usepackage{fullpage}
\usepackage{amsmath}
\usepackage{amsthm}
\usepackage{amssymb}
\usepackage{amsfonts,mathrsfs}
\usepackage{amsxtra}
\usepackage{epsfig}
\usepackage{verbatim}
\usepackage{color}
\usepackage{hyperref}
\usepackage{longtable}

\newtheorem{theorem}{Theorem}

\newtheorem{lemma}{Lemma}

\theoremstyle{remark}
\newtheorem*{remark}{Remark}

\newcommand{\BPol}{\mathbf{B}_{\Omega}^{\lambda}}

\newcommand{\BKol}{B_{\Omega}^{\lambda}}

\newcommand{\Ltol}{L^2(\Omega,\lambda)}
\newcommand{\Ltola}{L^2_a(\Omega,\lambda)}

\allowdisplaybreaks

\begin{document}

\title[The Bergman Projection Operator On Reinhardt Domains]{An Application of the Pr\'ekopa-Leindler Inequality and Sobolev Regularity of Weighted Bergman Projections}

\author{Yunus E. Zeytuncu}

\address{University of Michigan-Dearborn, Department of Mathematics and Statistics, Dearborn, MI 48128}

\email{zeytuncu@umich.edu}

\thanks{This work was partially supported by a grant from the Simons Foundation (\#353525), and also by a University of Michigan-Dearborn CASL Faculty Summer Research Grant.}

\begin{abstract}
We prove a general version of \cite[Theorem 4.1]{Boas84} to obtain Sobolev estimates for weighted Bergman projections on convex Reinhardt domains by using the Pr\'ekopa-Leindler inequality.
\end{abstract}

\maketitle
\section{Introduction}

\subsection{Setup}

Let $\Omega$ be a bounded domain in $\mathbb{C}^n$ and let $\lambda$ be a positive continuous function on $\Omega$. Let $\Ltol$ denote the space of square integrable functions on $\Omega$ with respect to the measure $\lambda(z)dV(z)$, where $dV(z)$ denotes the Lebesgue measure on $\mathbb{C}^n$. We denote the corresponding norm and inner product by $||\cdot||_{\Omega,\lambda}$ and $\left<\cdot,\cdot\right>_{\Omega,\lambda}$, respectively. The subspace of square integrable holomorphic functions is denoted by $\Ltola$. The  restriction on $\lambda$ guarantees that $\Ltola$ is a closed subspace and therefore the orthogonal projection operator (called the weighted Bergman projection) $\BPol:\Ltol\to\Ltola$ exists (see \cite{Pasternak90}). It follows from the Riesz representation theorem that $\BPol$ is an integral operator $$\BPol f(z)=\int_{\Omega}\BKol (z,w)f(w)\lambda(w)dV(w)$$
where the kernel $\BKol (z,w)$ is called the weighted Bergman kernel.

For any natural number $k$, the (weighted) $L^2$-Sobolev space $W^k(\Omega,\lambda)$ is a subspace of $\Ltol$ with the norm defined by

$$||f||_{k,\lambda}^2=\sum_{|\beta+\gamma|\leq k}\int_{\Omega}\left|\frac{\partial^{\beta+\gamma}}{\partial \overline{z}^{\beta}\partial z^{\gamma}}f(z)\right|^2\lambda(z)dV(z).$$
The values of $k$ for which $\BPol$ is bounded on $W^k(\Omega,\lambda)$ depend on the geometric and potential theoretic properties of the pair $(\Omega,\lambda)$. When $\BPol$ is bounded on $W^k(\Omega,\lambda)$ for all $k\geq 0$, we say $\BPol$ is exactly regular. For the case $\lambda\equiv 1$, we refer to \cite{BoasStraube96} for a comphrensive survey. For results in the weighted setting, we refer to \cite{ChangLi, BonamiGrellier, Ligocka89, Charpentier13, Charpentier14} and the references therein.

One of the well-developed cases is when $\Omega$ is a bounded smooth Reinhardt domain and $\lambda\equiv 1$. In this case, the radial symmetry and smoothness of the boundary are sufficient to conclude exact regularity (see \cite{Straube86} for a proof and \cite{Boas84} for two separate proofs). Under the additional convexity assumption, Boas gives an elementary proof of this fact in \cite{Boas84}. 
The convexity condition on $\Omega$ is used to establish a Brunn-Minkowski type inequality \cite[Lemma 5.2]{Boas84}. It turns out that this inequality is a special case of the Pr\'ekopa's inequality in \cite[Inequality 1.5]{Prekopa71}. The remaining arguments in Boas' proof do not require convexity and do hold on general Reinhardt domains. Therefore, a more general version of the inequality \cite[Lemma 5.2]{Boas84} would lead to a weighted version of \cite[Theorem 4.1]{Boas84}. In this note, we show that the Pr\'ekopa-Leindler inequality \cite[Theorem 7.1]{Gardner} serves this purpose and we obtain a generalization in the weighted setting.

\subsection{Statement}
Let $\Omega$ be a bounded smooth convex Reinhardt domain in $\mathbb{C}^n$. The weights that we consider are of the form $\lambda(z)=f\left(-\rho(z)\right)$, where $\rho$ is a smooth multi-radial (i.e., $\lambda(z_1,\cdots,z_n)=\lambda(|z_1|,\cdots,|z_n|)$) convex  defining function for $\Omega$ and $f:[0,\infty)\to\mathbb{R}$ is a positive, differentiable, and decreasing function. We also assume that $f(x)$ satisfies a convexity condition of the form, 
\begin{equation}\label{convexf}
\frac{\left(f(\frac{x+y}{2})\right)^2}{f(x)f(y)}\geq \delta
\end{equation}
for some $\delta>0$ and for all $x,y\in[0,\infty)$.

We note that the weights of the form above cover many of the settings in the literature. Indeed,
\begin{itemize}
\item if $f\equiv 1$ then we obtain the unweighted setting in \cite{Boas84},
\item if $\Omega$ is the unit disc in $\mathbb{C}^1$, $\rho(z)=|z|^2-1$, and $f(x)=x^a\exp\left(-\frac{b}{x^c}\right)$ for some $a\geq 1, b\geq 0$, and $c\geq 0$, then we obtain weights in \cite{Dostanic04, ZeytuncuSobolev},
\item if $\Omega$ is the unit ball in $\mathbb{C}^n$, $\rho(z)=||z||^2-1$, and $f(x)=x^a$ for some $a\geq 1$, then we are in the setting of \cite{FR, Bek},
\item  if $\Omega$ is the unit ball in $\mathbb{C}^n$, $\rho(z)=||z||^2-1$, and $f(x)=\exp\left(-\frac{1}{x}\right)$, then we recover \cite[Theorem 2]{ZeytuncuZeljko}.
\end{itemize}
Now we state the main result of this note.

\begin{theorem}\label{main}
Let $\Omega$ be a bounded smooth convex Reinhardt domain in $\mathbb{C}^n$ and $\lambda(z)=f\left(-\rho(z)\right)$ where $f$ and $\rho$ are as above.
%
%
Then $\BPol$ is bounded on $W^k(\Omega,\lambda)$ for all $k\in\mathbb{N}$.

\end{theorem}

%


The new ingredient in the proof of Theorem 2 is the Pr\'ekopa-Leindler inequality (stated below). Readers can find the history and more applications of this inequality in \cite{Gardner}.

\begin{theorem}{\cite[Theorem 7.1]{Gardner}}\label{PL}
Let $0<t<1$ and let $f, g$, and $h$ be nonnegative integrable functions on $\mathbb{R}^n$ satisfying
\begin{equation}\label{hfg}
h((1-t)x+ty)\geq f(x)^{1-t}g(y)^{t}
\end{equation}
for all $x,y\in\mathbb{R}^n$. Then
$$
\int_{\mathbb{R}^n}h(x)dx\geq \left(\int_{\mathbb{R}^n}f(x)dx\right)^{1-t}\left(\int_{\mathbb{R}^n}g(x)dx\right)^{t}~.
$$

\end{theorem}


\section{Proof of Theorem \ref{main}}

\subsection{Preliminaries}

First, we mention that on a bounded smooth convex Reinhardt domain one can always find a defining function $\rho(z)$ that is smooth multi-radial and convex. Indeed, in \cite{McNealConvex}, authors showed that on a smooth bounded convex domain, there exists a defining function that is convex in a neighborhood of the boundary.   Such a function can be extended to a defining function that is convex on $\Omega$ by using a gauge (Minkowski) function of the domain. The rotational symmetry can be also achieved by defining the function first on the radial image\footnote{The radial image $\mathcal{R}\subset\mathbb{R}^n_+$ of a Reinhardt domain $\Omega\subset\mathbb{C}^n$ is the image of $\Omega$ in the $|z_1|, \cdots, |z_n|$ variables in the real Euclidean space, i.e., the angular variables are suppressed.} of the domain.

Next, we note that the weights of the form above satisfy a convexity condition that will be needed when we invoke Theorem \ref{PL} below. Indeed, for $x,y\in\Omega$ we consider the following ratio:
\begin{align*}
\frac{\left(\lambda(\frac{x+y}{2})\right)^2}{\lambda(x)\lambda(y)}=\frac{\left(f\left(-\rho(\frac{x+y}{2})\right)\right)^2}{f\left(-\rho(x)\right)f\left(-\rho(y)\right)}~.
\end{align*}
Since $\rho$ is a convex function, we have 
$$\rho\left(\frac{x+y}{2}\right)\leq\frac{\rho(x)+\rho(y)}{2}~.$$ 
Furthermore, since $f$ is decreasing and $f$ satisfies \eqref{convexf} we obtain
\begin{align*}
\frac{\left(f\left(-\rho(\frac{x+y}{2})\right)\right)^2}{f\left(-\rho(x)\right)f\left(-\rho(y)\right)}\geq\frac{\left(f\left(\frac{-\rho(x)-\rho(y)}{2}
)\right)\right)^2}{f\left(-\rho(x)\right)f\left(-\rho(y)\right)}\geq\delta>0.
\end{align*}
Therefore, we get
\begin{equation}\label{convex}
\frac{\left(\lambda(\frac{x+y}{2})\right)^2}{\lambda(x)\lambda(y)}\geq \delta>0
\end{equation}
for all $x,y\in \Omega$.\\

%

Many parts of the proof have been already appeared in \cite{Boas84, ZeytuncuSobolev, ZeytuncuZeljko}; however, we repeat them here for completeness. The new ingredient (application of the Pr\'ekopa-Leindler inequality) is in the proof of Lemma \ref{coeff}.

For a multi-index $\gamma$, $d_{\gamma}$ denotes the  $L^2$-norm of the monomial $z^{\gamma}$, that is $$d_{\gamma}^2=\int_{\Omega}\left|z^{\gamma}\right|^2\lambda(z)dV(z).$$  The set of monomials $\left\{\frac{z^{\gamma}}{d_{\gamma}}\right\}$ forms an orthonormal basis of $\Ltola$ and the Bergman kernel $\BKol(z,w)$ is given by the sum $$\BKol(z,w)=\sum_{\gamma}\frac{z^{\gamma}\overline{w^{\gamma}}}{d_{\gamma}^2}.$$

For $j\in\mathbb{N}$, let $\mathbf{S}_j$ denote the truncation operator on $L^2_a(\Omega,\lambda)$; i.e. for $f(z)=\sum_{\alpha}f_{\alpha}z^{\alpha}$
$$\mathbf{S}_jf(z)=\sum_{|\alpha|\leq j}f_{\alpha}z^{\alpha}.$$
Note that $\mathbf{S}_j$ is a bounded operator with operator norm 1. Furthermore, for any holomorphic function $f(z)$ the truncation $\mathbf{S}_jf(z)$ is a polynomial and therefore is in $L^2_a(\Omega,\lambda)$. If the norms of the truncations $\mathbf{S}_jf(z)$ are uniformly bounded then $f(z)$ is also in $L^2_a(\Omega,\lambda)$.

\subsection{Proof of Theorem \ref{main}}
Our goal is to show that for a given multi-index $|\beta|\leq k$ and $f\in W^k(\Omega,\lambda)$, 
\begin{align}\label{derivative}
\left\|\frac{\partial^{\beta}}{\partial z^{\beta}}\BPol f\right\|_{\lambda}^2\lesssim ||f||_{k,\lambda}^2
\end{align}
where the constant is independent of  $f$.

Using the $\mathbf{S}_j$ operator above and the radial symmetry it is clear that 
$$\left\|\frac{\partial^{\beta}}{\partial z^{\beta}}\BPol f\right\|_{\lambda}^2=\lim_{j\to\infty}\left\|\mathbf{S}_j\frac{\partial^{\beta}}{\partial z^{\beta}}\BPol f\right\|_{\lambda}^2.$$
Therefore, it is enough to prove that 
\begin{align}\label{derivative2}
\left\|\mathbf{S}_j\frac{\partial^{\beta}}{\partial z^{\beta}}\BPol f\right\|_{\lambda_{}}^2\lesssim ||f||_{k,\lambda}^2
\end{align}
where the constant is independent of  $f$ and $j$. We will need the following integration by parts lemmas to obtain this inequality.

\begin{lemma}\label{operator}
For a given multi-index $\beta$ there exists a bounded operator  $M_{\beta}^{}$ on $L^2_a(\Omega_{},\lambda_{})$ such that 
$$ \left<h,\frac{\partial^{\beta}}{\partial z^{\beta}}g\right>_{\lambda_{}}= \left<\frac{\partial^{\beta}}{\partial z^{\beta}}M_{\beta}^{}h, g\right>_{\lambda_{}}$$
for all holomorphic polynomials $h$ and $g\in L^2_a(\Omega_{},\lambda_{})$. 
\end{lemma}

\begin{lemma}\label{parts}
For a given multi-index $\beta$ there exists a constant $K_{\beta}$ such that 
$$\left|\left< \frac{\partial^{\beta}}{\partial z^{\beta}}h,f \right>_{\lambda_{}}\right|\leq K_{\beta}||h||_{\lambda_{}}||f||_{|\beta|,\lambda}$$
for  all $f\in W^{|\beta|}(\Omega,\lambda)$ and all holomorphic polynomials $h$.
\end{lemma}

Assuming these two lemmas for now, we obtain \eqref{derivative2} as follows. Let $h\in L^2_a(\Omega,\lambda)$, 
\begin{align*}
\left\|\mathbf{S}_j\frac{\partial^{\beta}}{\partial z^{\beta}}\BPol f\right\|_{\lambda_{}}^2&=\sup_{||h||_{\lambda}\leq 1}
\left| \left<h, \mathbf{S}_j \frac{\partial^{\beta}}{\partial z^{\beta}}\BPol f \right>_{\lambda_{}}\right|^2.
\end{align*}
On the other hand, 
\begin{align*}
\left| \left<h, \mathbf{S}_j \frac{\partial^{\beta}}{\partial z^{\beta}}\BPol f \right>_{\lambda_{}}\right|
&=\left| \left<\mathbf{S}_jh, \frac{\partial^{\beta}}{\partial z^{\beta}}\BPol f \right>_{\lambda_{}}\right|\\
&=\left| \left<  \frac{\partial^{\beta}}{\partial z^{\beta}}M_{\beta}^{}\mathbf{S}_jh,\BPol f
 \right>_{\lambda_{}}\right| ~ \text{ by Lemma \ref{operator}}\\
 &=\left| \left<  \frac{\partial^{\beta}}{\partial z^{\beta}}M_{\beta}^{}\mathbf{S}_jh,f
 \right>_{\lambda_{}}\right| ~ \text{ by self adjointness}\\
 &\lesssim ||M_{\beta}\mathbf{S}_jh||_{\lambda_{}}~||f||_{k,\lambda}~\text{ by Lemma \ref{parts}}\\
&\lesssim ||\mathbf{S}_jh||_{\lambda_{}}~||f||_{k,\lambda}~\text{ by Lemma \ref{operator}}\\
&\leq ||h||_{\lambda}~||f||_{k,\lambda}
\end{align*}
where the constants are clearly independent of $j$ and $f$. In order to finish the proof of Theorem \ref{main}, it remains to prove the lemmas above.

\subsection{Proof of Lemma \ref{parts}} The same lemma has appeared in \cite[Lemma 21]{ZeytuncuZeljko}, which is in fact a rendition of \cite[Lemma 2.1]{Straube86} and \cite[Lemma 6.1]{Boas84}. We go over the arguments briefly for completeness, and refer to \cite{ZeytuncuZeljko} for more details.

If $f$ is supported on a compact subset of $\Omega$ then the estimate follows easily from the Bergman inequality. Hence we assume $f$ is supported in a small neighborhood of the boundary of $\Omega$. Furthermore, we choose this neighborhood such that we can find smooth orthonormal vector fields $L_1,\cdots, L_n$ with the property that $L_1,\cdots, L_{n-1}$ and $L_n+\overline{L}_n$ are tangent to the boundary of $\Omega$. As a consequence of the Cauchy-Riemann equations, we can write any derivative of a holomorphic polynomial $h$ in terms of these vector fields; indeed, there exist $c_{ij}\in C^{\infty}({\overline{\Omega}})$ such that 
$$\frac{\partial}{\partial z_j}h=\left(\sum_{i=1}^{n-1}c_{ij}L_i+c_{nj}(L_n+\overline{L_n})\right)h=:\mathcal{L}_jh~.$$
Also note that by the choice of $\lambda=f\left(-\rho\right)$, $T(\lambda)=0$ for any tangential vector field $T$. This means, we have
\begin{align*}
\left< \frac{\partial}{\partial z_j}h, f\right>_{\lambda}&=\left< \mathcal{L}_j(h), f\right>_{\lambda}~\text{ for some tangential vector field } \mathcal{L}_j\\ 
&=\left< \mathcal{L}_j(h)\lambda, f\right>\\
&=\left< \mathcal{L}_j(h\lambda), f\right>~\text{ since }\mathcal{L}_j(\lambda)=0\\
&=\left< h\lambda,  \widetilde{\mathcal{L}_j}(f)\right>~\text{ no boundary terms since }\mathcal{L}_j\text{ is tangential }\\
&=\left< h, \widetilde{\mathcal{L}_j}(f)\right>_{\lambda}~
\end{align*}
where $\widetilde{\mathcal{L}_j}$ is a first order differential operator with $C^{\infty}({\overline{\Omega}})$ coefficients. For a multi-index $\beta$, if we iterate this argument $|\beta|$ times we get,
\begin{align*}
\left< \frac{\partial^{\beta}}{\partial z^{\beta}}h, f\right>_{\lambda}
=\left< h, \widetilde{\mathcal{L}_{\beta}}(f)\right>_{\lambda}
\end{align*}
for some differential operator $\widetilde{\mathcal{L}_{\beta}}$ of order $|\beta|$ with $C^{\infty}(\overline{\Omega})$ coefficients. We conclude the proof of Lemma \ref{parts} by the H\"older's inequality.

\subsection{Proof of Lemma \ref{operator}} This lemma is stated in \cite[Lemma 4.2]{Boas84} for $\lambda\equiv 1$. We define $M_{\beta}^{}$ as follows. For a monomial $z^{\alpha}$, we set

$$M_{\beta}^{}(z^{\alpha})=\frac{(\alpha+\beta)!(\alpha+\beta)!}{\alpha!(\alpha+2\beta)!} \frac{d_{ \alpha}^2}{d_{ \alpha+\beta}^2}z^{\alpha+2\beta}~.$$
The explicit expression of $M_{\beta}$ is imposed by the orthonormality of the set $\left\{\frac{z^{\gamma}}{d_{\gamma}}\right\}$, and the point of the lemma is to prove the continuity of $M_{\beta}$.

For this purpose, we compute the norms $$\frac{||M_{\beta}^{}(z^{\alpha})||_{\lambda_{}}}{||z^{\alpha}||_{\lambda_{}}}=\frac{(\alpha+\beta)!(\alpha+\beta)!}{\alpha!(\alpha+2\beta)!}\frac{d_{ \alpha}d_{ \alpha+2\beta}}{d_{ \alpha+\beta}^2}$$
the first fraction is uniformly bounded since
\begin{align*}
\frac{(\alpha+\beta)!(\alpha+\beta)!}{\alpha!(\alpha+2\beta)!}=\frac{\binom{\alpha+\beta}{\beta}}{\binom{\alpha+2\beta}{\beta}}\leq 1.
\end{align*}

\noindent It remains to prove that the second fraction 
$$ \frac{d_{ \alpha}d_{ \alpha+2\beta}}{d_{ \alpha+\beta}^2} $$
is uniformly bounded. This is a consequence of the Pr\'ekopa-Liendler inequality as explained in the next lemma.

\begin{lemma}\label{coeff}
For a given multi-index $\beta$ there exists a constant $K_{\beta}$ such that 
\begin{align}\label{key}
d_{ \alpha}d_{ \alpha+2\beta}\leq K_{\beta}(d_{ \alpha+\beta})^2
\end{align}
for all multi-indices $\alpha$.
\end{lemma}

\begin{proof}
We want to show that there exists $K_{\beta}>0$ such that 
\begin{equation}\label{coeffi}
\int_{\Omega}|z^{\alpha}|^2\lambda(z)dV(z)\int_{\Omega}|z^{\alpha+2\beta}|^2\lambda(z)dV(z)\leq K_{\beta}\left(\int_{\Omega}|z^{\alpha+\beta}|^2\lambda(z)dV(z)\right)^2~.
\end{equation}

Recall that $\Omega$ is a Reinhardt domain and all the functions inside the integrals are multi-radial; therefore, the integration is taking place on the radial image $\mathcal{R}$ of $\Omega$.  We use $r$ to denote the vector in $\mathbb{R}^n$. We define three functions
\begin{align*}
h(r)&=r^{\zeta+\eta}\lambda(r)\chi_{\mathcal{R}}(r)\\
f(r)&=r^{\zeta}\lambda(r)\chi_{\mathcal{R}}(r)\\
g(r)&=r^{\zeta+2\eta}\lambda(r)\chi_{\mathcal{R}}(r)
\end{align*}
where $\zeta$ and $\eta$ are multi-indices and $\chi_{\mathcal{R}}(r)$ stands for the characteristic function of the convex set $\mathcal{R}$. We verify the condition \eqref{hfg} in Theorem \ref{PL}. Namely, we claim
$$K_{\eta}h\left(\frac{x+y}{2}\right)\geq \sqrt{f(x)g(y)}$$
for all $x,y\in\mathcal{R}$ and for some $K_{\eta}>0$. That is, we claim
$$K_{\eta}\left(\frac{x+y}{2}\right)^{\zeta+\eta}\lambda\left(\frac{x+y}{2}\right)\chi_{\mathcal{R}}\left(\frac{x+y}{2}\right)\geq \sqrt{x^{\zeta}y^{\zeta+2\eta}\lambda(x)\lambda(y)\chi_{\mathcal{R}}(x)\chi_{\mathcal{R}}(y)}.$$
We can eliminate $\chi_{\mathcal{R}}$'s from both sides since the set $\mathcal{R}$ is convex. We can also eliminate $\lambda$'s from both sides by the assumption \eqref{convex}. Therefore it remains to show that 
$$K_{\eta}\left(\frac{x+y}{2}\right)^{\zeta+\eta}\geq \sqrt{x^{\zeta}y^{\zeta+2\eta}}.$$
Furthermore, it is enough to prove this in one dimension since the multi dimensional version follows by iteration. 

In other words, we claim that for $u,v\geq0$ and $a,b\in\mathbb{N}$, there exists $K>0$ (independent of $a$) such that 
$$K\left(\frac{u+v}{2}\right)^{a+b}\geq \sqrt{u^av^{a+2b}}.$$
We rewrite this inequality as follows
\begin{align*}
K\left(\frac{u+v}{2}\right)^{a+b}&\geq \sqrt{u^av^{a+2b}}\\
K\left(\frac{u+v}{2}\right)^{a}\left(\frac{u+v}{2}\right)^{b}&\geq \sqrt{u^av^{a}}v^b\\
K\left(\frac{\frac{\sqrt{u}}{\sqrt{v}}+\frac{\sqrt{v}}{\sqrt{u}}}{2}\right)^{a}\left(\frac{\frac{u}{v}+1}{2}\right)^{b}&\geq1\\
\end{align*}
If we set $t=\frac{\sqrt{u}}{\sqrt{v}}$ and observe that $\frac{t+t^{-1}}{2}\geq 1$ on $(0,\infty)$, then we conclude the desired inequality by choosing $K=2^b$.

This means we can invoke Theorem \ref{PL} for the functions $h, f, g$ and appropriate multi-indices $\zeta$ and $\eta$ to obtain \eqref{coeffi}. This concludes the proof of Lemma \ref{coeff} and the proof of Theorem \ref{main}. 
\end{proof}

\begin{remark}
In the unweighted setting (i.e. $\lambda\equiv 1$) the exact regularity holds without any convexity or pseudoconvexity assumption as demonstrated in \cite{Straube86}. Therefore, it is plausible to think that in the weighted setting, the radial symmetry of the domain and weight should be sufficient to conclude boundedness on Sobolev spaces. However, our proof of Theorem \ref{main} does require convexity of the domain and weight, since we invoke the Pr\'ekopa-Leindler inequality. It will be a curious project to investigate whether one can drop the convexity (or even the pseudoconvexity) assumption on $\Omega$ and obtain regularity for arbitrary weights.\\
\end{remark}

\section*{Acknowledgment} 
I'd like to thank anonymous referees for constructive comments. I'd also like to thank Bo Berndtsson for his enlightening lectures on complex Brunn-Minkowski theory at the Analysis and Geometry in Several Complex Variables Conference at Texas A\&M University at Qatar. I'd also like to thank the organizers of this conference, Shiferaw Berhanu, Nordine Mir and Emil J. Straube for a memorable time in Qatar.


\bibliographystyle{ams}

\end{document}